\documentclass[11 pt]{amsart}
\usepackage{amscd,amsfonts,amssymb,amsmath}
\usepackage{hyperref}
\usepackage{epsfig}
\newtheorem{theorem}{Theorem}[section]
\newtheorem{corollary}[theorem]{Corollary}
\newtheorem{lemma}[theorem]{Lemma}
\newtheorem{proposition}[theorem]{Proposition}

\theoremstyle{definition}

\newtheorem{conjecture}[theorem]{Conjecture}

\newtheorem{definition}[theorem]{Definition}

\newtheorem{remark}[theorem]{Remark}
\numberwithin{equation}{subsection}
\newtheorem*{ack}{Acknowledgement}
\usepackage[all,cmtip]{xy}
\newcommand{\Ker}{\operatorname{Ker}}
\newcommand{\As}{\operatorname{As}}

\newcommand{\Conj}{\operatorname{Conj}}

\newcommand{\C}{\operatorname{C}}

\newcommand{\Inn}{\operatorname{Inn}}

\newcommand{\Z}{\operatorname{Z}}

\usepackage[autostyle]{csquotes}

\begin{document}
\title{Link quandles are residually finite}

\author{Valeriy G. Bardakov}
\author{Mahender Singh}
\author{Manpreet Singh}

\address{Sobolev Institute of Mathematics, 4 Acad. Koptyug avenue, 630090, Novosibirsk, Russia.}
\address{Novosibirsk State  University, 2 Pirogova Street, 630090, Novosibirsk, Russia.}
\address{Novosibirsk State Agrarian University, Dobrolyubova street, 160, Novosibirsk, 630039, Russia.}
\email{bardakov@math.nsc.ru}

\address{Department of Mathematical Sciences, Indian Institute of Science Education and Research (IISER) Mohali, Sector 81,  S. A. S. Nagar, P. O. Manauli, Punjab 140306, India.}
\email{mahender@iisermohali.ac.in}
\email{manpreetsingh@iisermohali.ac.in}


\begin{abstract}
Residual finiteness is known to be an important property of groups appearing in combinatorial group theory and low dimensional topology. In a recent work \cite{Bardakov-Singh-Singh} residual finiteness of quandles was introduced, and it was proved that free quandles and knot quandles are residually finite. In this paper, we extend these results and prove that free products of residually finite quandles are residually finite provided their associated groups are residually finite. As associated groups of link quandles are link groups, which are known to be residually finite, it follows that link quandles are residually finite.
\end{abstract}

\maketitle

\section{Introduction}
A quandle is a non-empty set with a binary operation that satisfies axioms modelled on the three Reidemeister moves of diagrams of knots in $\mathbb{S}^3$. These objects first appeared in the work of Joyce \cite{Joyce} under the name  quandle, and that of Matveev \cite{Matveev} under the name distributive groupoid. They independently proved that one can associate a quandle $Q(L)$ to each tame link $L$ that is an invariant of links. Further, if $L_1$ and $L_2$ are two non-split tame links with $Q(L_1) \cong Q(L_2)$, then there is a homeomorphism of $\mathbb{S}^3$ mapping $L_1$ onto $L_2$, not necessarily preserving the orientations of the ambient space. Besides knot theory, quandles have shown appearance in various areas of mathematics, and have been a subject of intensive investigation in recent years. The reader is referred to the survey articles \cite{Carter, Kamada, Nelson} for more on recent developments in the subject.
\vspace*{.3mm}

Although link quandles are good invariant for tame links, it is usually difficult to check whether two quandles are isomorphic.  This motivates search for newer properties of quandles, particularly, of link quandles. It is well-known that residual finiteness and other residual properties play a crucial role in combinatorial group theory and low dimensional topology. Investigation of residual finiteness of link groups, in general 3-manifold groups, has been of interest for a long time. Neuwirth \cite{Neuwirth} showed that knot groups of fibered knots are residually finite. Mayland \cite{Mayland} extended the result to groups of twist knots, and Stebe \cite{Stebe} to certain class of non-fibered knots. As a consequence of the proof of the geometrization conjecture due to Perelman \cite{Perelman1, Perelman2, Perelman3}, all finitely generated 3-manifold groups, in particular link groups, have been shown to be  residually finite \cite{Hempel}. See the recent memoir \cite{Aschenbrenner-Friedl} for more on this theme. Since groups are rich sources of quandles, many ideas from group theory have been brought to the realm of quandles. This motivated the recent work \cite{Bardakov-Singh-Singh}, where we initiated study of residual finiteness of quandles. It was proved that free quandles and knot quandles of tame knots are residually finite. However, residual finiteness of link quandles remained unsettled. The purpose of this paper is to prove that link quandles of tame links are also residually finite. For non-split links this is established by extending arguments of \cite{Bardakov-Singh-Singh} and using a result of Long and Niblo \cite{Long-Niblo} on finite separability of $\pi_1(S, p)$ in $\pi_1(M, p)$, where $M$ is an orientable irreducible compact 3-manifold and $S$ an incompressible connected subsurface of a component of the boundary of $M$ containing the
base point p. For split links we first prove that  free products of residually finite quandles are residually finite provided their associated groups are residually finite. The result then follows by observing that link quandles of arbitrary links are free products of quandles of their non-split components, and that associated groups of link quandles are the corresponding link groups. As a consequence, we deduce that link quandles are Hopfian and have solvable word problem.
\vspace*{.3mm}

The paper is organised as follows. We begin by setting the necessary background in Section \ref{prelim}. In Section \ref{sec-non-split}, we prove that link quandles of non-split links are residually finite (Theorem \ref{non-split-link-quandle-is-residually-finite}). In Section \ref{sec-split-link}, we prove that link quandles of split links are residually finite (Theorem \ref{split-link-res-finite}). This is achieved by first proving that free products of residually finite quandles are residually finite if their associated groups are residually finite (Theorem \ref{arbitrary-free-product-of-residually-finite-quandles}). Finally, in Section \ref{concluding-remarks}, using a recent result of Bardakov and Nasybullov \cite{Bardakov-Nasybullov} on embedding of free products of quandles into their enveloping groups, we give a short proof of Theorem \ref{split-link-res-finite} for links with untangled components where each component is a prime knot. We prove that associated groups of finite quandles are residually finite (Proposition \ref{ass-group-finite}). Since associated groups of free quandles and link quandles are also residually finite, we conclude with a conjecture (\ref{residually-finite-associated-group}) that associated group of any finitely presented residually finite quandle is residually finite.
\vspace*{.3mm}
Throughout the paper, all knots and links are tame.
\bigskip

\section{Basic definitions and examples}\label{prelim}
We begin with the definition of a quandle.
\vspace*{.3mm}

\begin{definition}
A {\it quandle} is a non-empty set $X$ with a binary operation $(x,y) \mapsto x * y$ satisfying the following axioms:
\begin{enumerate}
\item $x*x=x$ for all $x \in X$;
\item For any $x,y \in X$ there exists a unique $z \in X$ such that $x=z*y$;
\item $(x*y)*z=(x*z) * (y*z)$ for all $x,y,z \in X$.
\end{enumerate}
\end{definition}

\vspace*{.3mm}
A non-empty set with a binary operation satisfying only the axioms (2) and (3) is called a {\it rack}. Obviously, every quandle is a rack, but not conversely. A quandle $X$ is called {\it trivial} if $x*y=x$ for all $x, y \in X$. A trivial quandle can contain arbitrary number of elements. Note that the quandle axioms are equivalent to saying that for each $x \in X$, the map $S_x : X \to X$ given by $S_x(y) = y *x$ is an automorphism of $X$ fixing $x$. These maps are referred as {\it inner automorphisms} of $X$, and the group generated by them is denoted by $\Inn(X)$. The second quandle axiom is equivalent to saying that there exists dual binary operation on $X$, written as $(x, y) \mapsto x \ast^{-1} y$, and satisfying $$x \ast y = z~\textrm{if and only if}~ x = z \ast^{-1} y$$ for all $x, y, z \in X$.
\vspace*{.3mm}

Although links are rich sources of quandles, many interesting examples of quandles come from groups.
\begin{itemize}
\item A group $G$ equipped with the binary operation $a*b= b^{-1} a b$ gives a quandle structure on $G$, called the {\it conjugation quandle}, and denoted $\Conj(G)$.

\item If $G$ is a group, $z$ an element of $G$ and $H$ a subgroup of the centralizer $\C_G(z)$ of $z$ in $G$, then the set of right cosets of $H$ in $G$ becomes a quandle by defining $$H x  \ast H y = H z ^{-1} x y^{-1} z y.$$ The quandle so obtained is denoted by $(G,H,z)$.

\item The preceding example can be generalized. Let $\{z_i~|~i \in I \}$ be elements of a group $G$, and $\{H_i~|~i \in I \}$ subgroups of $G$ such that $H_i \leq \C_{G}(z_i)$ for all $i \in I$. Then the disjoint union $Q=\sqcup_{i \in I}(G,H_i,z_i)$ becomes a quandle with
$$
H_i x  \ast H_j y= H_i z_i^{-1} x y^{-1} z_j y.
$$
\end{itemize}
\vspace*{.3mm}

In the reverse direction each quandle give rise to a group as follows.

\begin{definition}
The \textit{associated group} $\As(X)$ of a quandle $X$ is defined to be the group generated by the set $\{e_x~|~x \in X\}$ modulo the relations $e_{x \ast y} = e_y ^{-1} e_x e_y$ for all $x , y \in X$. \end{definition}
\vspace*{.6mm}

The presentation of the associated group of a quandle can be reduced as follows \cite[Theorem 5.1.7]{Winker}.

\begin{theorem}\label{presentation-of-associated-group}
If $Q$ is a quandle with a presentation $Q=  \langle X~~|~~R \rangle$, then its associated group has presentation $\As(Q) \cong  \langle e_x, ~~x \in X~~|~~\bar{R} \rangle$, where $\bar{R}$ consists of relations in $R$ with each expression $x \ast y$ replaced by $e_{y}^{-1} e_x e_y$.
\end{theorem}

It is well-known that the associated group of the link quandle $Q(L)$ of a link $L$ is the link group $\pi_1(\mathbb{S}^3 \setminus L)$, and the associated group of the free quandle on a set $X$ is the free group on $X$. For a given quandle $X$, there is a natural map
$$\eta: X \rightarrow \As(X)$$
defined as $\eta(x) = e_x$, which  is a quandle homomorphism considering the associated group $\As(X)$ as the conjugation quandle $\Conj\big(\As(X)\big)$. A quandle homomorphism $f: X \rightarrow Y$ induces a group homomorphism $f_{\sharp} : \Conj\big(\As(X)\big) \rightarrow \Conj\big(\As(Y)\big)$ defined by $f_{\sharp}(e_x) = e_{f(x)}.$ Moreover, there is a group homomorphism
$$\psi_X : \As(X) \rightarrow \Inn(X)$$
defined as $\psi_X(e_x) = S_x$ where $x \in X$, $e_x \in \As(X)$ and $S_x \in \Inn(X)$. It is easy to see that $\Ker(\psi_X)$ is contained in the center of the associated group $\As(X)$ giving rise to the central extension
$$ 1 \longrightarrow \Ker(\psi_X) \longrightarrow \As(X) \longrightarrow \Inn(X) \longrightarrow 1$$
 of groups. Notice that the homomorphism $\psi_X$ induces a right action of the associated group $\As(X)$ on the quandle $X$ defined as
$$x. e_y = x \ast y,$$
which we shall use later.
\vspace*{.6mm}

\begin{remark}\label{image-under-eta}
A trivial quandle homomorphism $f : X \rightarrow \{a\}$ induces a group homomorphism $f_{\sharp}: \As(X) \rightarrow \As(\{a\}) \cong \mathbb{Z}$, where $f_{\sharp}(e_x) = 1$ for all $x \in X$. Thus, under the natural map $\eta: X \rightarrow \As(X)$ none of the elements of $X$ map to the identity of the associated group $\As(X)$.
\end{remark}
\bigskip

\section{Link quandles of non-split links}\label{sec-non-split}
Residual finiteness of quandles was introduced and investigated in \cite{Bardakov-Singh-Singh}.
\vspace*{.3mm}

\begin{definition}
A quandle $Q$ is said to be {\it residually finite} if for $x,y \in X$ and $x \neq y$, there exist a finite quandle $F$ and a quandle homomorphism $\phi: X \rightarrow F$ such that $\phi(x) \neq \phi(y)$.
\end{definition}
\vspace*{.3mm}

Obviously, every finite quandle and every trivial quandle is residually finite. See \cite{Bardakov-Singh-Singh} for more examples.  A more general notion is that of a finitely separable subgroup of a group.

\begin{definition}
A subgroup $H$ of a group $G$ is said to be {\it finitely separable} if for any $g \in G \setminus H$, there exist a finite group $F$ and a group homomorphism $\phi: G \rightarrow F$ such that $\phi(g) \not\in \phi(H)$.
\end{definition}
\par

For example, if $G$ is a residually finite group and $H$ a finite subgroup of $G$, then $H$ is finitely separable in $G$. Recall that a connected $3$-manifold $M$ is said to be {\it irreducible} if every embedded $2$-sphere in $M$ bounds a 3-ball in $M$. The following result concerning  finitely separable subgroups of fundamental groups of irreducible 3-manifolds is due to Long and Niblo \cite{Long-Niblo}.
\vspace*{.3mm}

\begin{theorem}\label{Long-Niblo-thm}
Suppose that $M$ is an orientable, irreducible compact 3-manifold and $X$ an incompressible connected subsurface of a component of $\partial(M)$. If $p\in X$ is a base point, then $\pi_1(X,p)$ is a finitely separable subgroup of $\pi_1(M,p)$.
\end{theorem}
\vspace*{.3mm}

We begin with the following result which will be used in the sequel.
\par

\begin{proposition}\label{res-finite-quandle}
Let $G$ be a group, $\{z_i~|~i \in I \}$ be a finite set of elements of $G$, and $\{H_i~|~i \in I \}$ subgroups of $G$ such that $H_i \le \C_G(z_i)$. If each $H_i$ is finitely separable in $G$, then the quandle $\sqcup_{i \in I}(G,H_i,z_i)$ is residually finite.
\end{proposition}

\begin{proof}
Let $H_k a \neq	  H_j b$ be two elements of $\sqcup_{i \in I}(G,H_i,z_i)$.
\vspace*{.3mm}

Case 1: $k \neq j$. Let $F=\{a',b' \}$ be a two element trivial quandle, and define $$\phi: \sqcup_{i \in I}(G,H_i,z_i) \rightarrow F$$ as
$$
\phi(H_i x)= \left\{
\begin{array}{ll}
a' &  \textrm{if}~ i=k,\\
b'~& \textrm{if}~ i\neq k.
\end{array} \right.
$$
Then $\phi$ is a quandle homomorphism with $\phi(H_k a)\neq \phi(H_j b)$.
\vspace*{.3mm}

Case 2: $k=j$. Since $H_k a \neq H_k b$, $a \neq hb$ for any $h \in H_k$. Further, since $H_k$ is finitely separable in $G$, there exists a finite group $F$ and a group homomorphism $\phi:G \rightarrow F$ such that $\phi(a) \neq \phi(h b)$ for each $h \in H_k$. Let $\overline{H_i}:= \phi(H_i)$  and $\bar{z_i}:=\phi(z_i)$ for each $i \in I$. Then  $\overline{H_i} \le  \C_F( \bar{z_i})$, and $\sqcup_{i \in I}(F, \overline{H_i},\bar{z_i})$ is a finite quandle. Further, the group homomorphism $\phi:G \rightarrow F$ induces a map $$\bar{\mathbb{\phi}}:\sqcup_{i \in I}(G,H_i,z_i) \rightarrow \sqcup_{i \in I}(F,\overline{H_i},\bar{z_i})$$ given by $$\bar{\mathbb{\phi}}(H_i x)=\overline{H_i} \phi(x),$$ which is a quandle homomorphism. Also, $\bar{\mathbb{\phi}}(H_k a) \neq \bar{\mathbb{\phi}}(H_k b)$, otherwise $\phi(a)=\phi(hb)$ for some $h\in H_k$, which is a contradiction.  Hence, $\sqcup_{i \in I}(G,H_i,z_i)$ is residually finite.
\end{proof}
\bigskip

To prove the residual finiteness of quandles of non-split links, we first recall the general construction of link quandles.  Let $L$ be an oriented link in $\mathbb{S}^3$ with components $K_1, K_2, \ldots, K_t$. Let $V(L)$ be a tubular neighborhood of $L$ and $C(L)=\overline{\mathbb{S}^{3}\setminus V(L)}$. Clearly, $V(L)$ is the disjoint union $V(K_1) \sqcup V(K_2) \sqcup \ldots \sqcup V(K_t)$, where $V(K_i)$ is a tubular neighborhood of $K_i$. Fix a base point $x_0$ in $C(L)$. Let $Q(L)$ be the set of homotopy classes of paths in $C(L)$ with initial point on the boundary of $\partial \big(V(L)\big)$ of $V(L)$ and end point at $x_0$. Then the binary operation
$$[a] \ast [b] :=[a b^{-1} m_{b(1)} b],$$
where $m_{b(1)}$ is a meridian at point $b(1)$, turns $Q(L)$ into a quandle, called the link quandle of $L$. For each $i $, define $Q(K_i)$ to be the set of homotopy classes of paths in $C(L)$ starting on the boundary $\partial V(K_i)$ of a tubular neighborhood $V(K_i)$ of $K_i$ and ending at $x_0$. Then each $Q(K_i)$ is a subquandle of $Q(L)$ with the above operation, and in fact $Q(L)=\sqcup_{i \in I} Q(K_i)$.
\par
Let $G(L)=\pi_1 \big(C(L),x_0 \big)$ be the link group of $L$. Then $G(L)$ acts on $Q(L)$ as
$$
[a]. [\alpha]=[a \alpha],
$$
where $[\alpha] \in G(L)$ and $[a] \in Q(L)$. One can easily check that the action keeps each $Q(K_i)$ invariant.
\vspace*{.3mm}

For each $i$, let $x_i \in \partial V(K_i)$ be a fixed base point, and $s_i$ a path from $x_i$ to $x_0$. Then each
$$\hat{s_i}: \pi_1(\partial V(K_i),x_i) \rightarrow G(L)$$
defined as  $\hat{s_i}([\alpha])=[s_i ^{-1} \alpha s_i]$
is a group homomorphism. If $H_i$ denote the image of $\hat{s_i}$, then we have the following result whose proof is analogous to the one worked out in \cite[Lemma 2]{Matveev} for knots.

\begin{lemma}\label{transitive-action}
The action of $G(L)$ on $Q(K_i)$ is transitive and stabilizer of $[s_i]$ is $H_i$.
\end{lemma}
\par

Let $m_i$ be the image of the meridian in $\pi_1(\partial V(K_i), x_i)$ under the map $\hat{s_i}$. Then $\sqcup_{i \in I}\big(G(L) ,H_i,m_i \big)$ becomes quandle under the operation defined as $$H_ig \ast H_j g'=H_i g g'^{-1} m_j g'.$$

\begin{theorem}\label{non-split-link-quandle-is-residually-finite}
The link quandle of a non-split link is residually finite.
\end{theorem}

\begin{proof} First note that, for each $i$, the map $\big(G(L) ,H_i,m_i \big) \rightarrow Q(L_i)$ given by $$H_i g\mapsto [s_i g]$$ is bijective (by Lemma \ref{transitive-action}), and is also a quandle homomorphism. Since $Q(L)=\sqcup_{i \in I} Q(K_i)$, we obtain an isomorphism of quandles
$$\sqcup_{i \in I} \big(G(L),H_i,m_i \big) \rightarrow Q(L).$$ As $L$ is non-split, it follows from Theorem \ref{Long-Niblo-thm} that each $H_i$ is finitely separable in $G(L)$. Thus, by Proposition \ref{res-finite-quandle}, the link quandle $Q(L)$ is residually finite.
\end{proof}
\vspace*{.3mm}

The preceding theorem generalizes \cite[Theorem 6.8]{Bardakov-Singh-Singh} to links. It must be noted that the above arguments do not work for split links since their complements are reducible 3-manifolds. However, we give an algebraic proof for this case in the next section.
\bigskip

\section{Free products and quandles of split links}\label{sec-split-link}
We define the free product of quandles as follows. Let
$$
A = \langle X~|~R \rangle ~\textrm{and}~B = \langle Y~|~S \rangle
$$
be two quandles with non-intersecting sets of generators. Then the free product $A \star B$ is a quandle that is defined by the presentation
$$
A \star B = \langle X \sqcup Y~|~R \sqcup S\rangle.
$$
For example, if  $FQ_n$ is the free $n$-generated quandle, then
$$
FQ_n = \underbrace{T_1 \star T_1 \star \cdots \star T_1}_{n~\textrm{copies}},
$$
the free product of $n$ copies of trivial one element quandles. We refer the reader to the recent work \cite[Section 7]{Bardakov-Nasybullov} for more on free products of quandles. Free product of racks can be defined analogously.
\vspace*{.3mm}

\begin{lemma}\label{presentation-of-associated-group-of-free-product-of-quandles}
If $Q_1, Q_2$ are quandles, then $\As(Q_1 \star Q_2) \cong  \As(Q_1) \star \As(Q_2).$
\end{lemma}

\begin{proof}
If $Q_1$ and $Q_2$ have presentations $Q_1 = \langle X_1 ~~|~~R_1 \rangle$ and $Q_2 = \langle X_2 ~~|~~R_2 \rangle$, then $Q_1 \star Q_2 = \langle X_1 \sqcup X_2 ~~|~~R_1 \sqcup R_2 \rangle$. Now, by Theorem \ref{presentation-of-associated-group}
\begin{eqnarray*}
\As(Q_1 \star Q_2) & \cong & \langle e_x ~~(x \in  X_1 \sqcup X_2 )~~|~~\bar{R_1} \sqcup \bar{R_2} \rangle \\
 & \cong &  \langle e_x~~(x \in X_1) ~~|~~\bar{R_1} \rangle \star  \langle e_x ~~(x \in X_2) ~~|~~\bar{R_2} \rangle\\
 & \cong &  \As(Q_1) \star \As(Q_2).
\end{eqnarray*}
\end{proof}

The following result is well-known in combinatorial group theory, first proved by Gruenberg \cite[Theorem 4.1]{Gruenberg}. See also \cite{Baumslag-Tretkoff, Cohen}.
\vspace*{.3mm}

\begin{theorem}\label{free-product-of-residually-finite-groups}
A free product of residually finite groups is residually finite.
\end{theorem}
\vspace*{.6mm}

We prove an analogue of the preceding theorem for quandles provided their associated groups are residually finite. Throughout, for ease of notation, for elements $x_0,x_1, x_2, \ldots, x_n$ of a quandle $X$, we write $x_0 \ast^{e_1} x_1 \ast^{e_2} x_2 \ast ^{e_3} \cdots \ast ^{e_n} x_n$ to denote the element $( \cdots(( x_0 \ast^{e_1} x_1) \ast^{e_2} x_2) \ast ^{e_3} \cdots )\ast ^{e_n} x_n$, where $e_i \in \{1,-1\}$. We note that every element of a quandle $X$ can be written in this form. Moreover, the expression $x_0 \ast^{e_1} x_1 \ast^{e_2} x_2 \ast ^{e_3} \cdots \ast ^{e_n} x_n$ is called a {\it reduced form} when $x_0 \neq x_1$ and if $x_i =x_{i+1}$, then $e_i = e_{i+1}$. Notice that the reduced form is not unique. For example, if $Q = \{ t \} \star R_3$ is the free product of one element trivial quandle and the dihedral quandle $R_3 = \{ a_0, a_1, a_2 \}$, then
$$
(t * a_1) * a_2 = (t * a_2) * (a_1 * a_2) =  (t * a_2) * a_0.
$$
\vspace*{.6mm}

\begin{theorem}\label{finite-free-product-of-residually-finite-quandles}
Let $Q_1,Q_2, \ldots,Q_n$ be residually finite quandles. If each associated group $\As(Q_i)$ is residually finite, then $Q_1 \star Q_2 \star \cdots \star Q_n$ is a residually finite quandle.
\end{theorem}
\begin{proof}
It is enough to consider the case $n=2$. Set $Q= Q_1 \star Q_2$. Let $x$ and $x'$ be two distinct elements of $Q$, where
\begin{align*}
x &=a_0 \ast ^{e_1}a_1  \ast ^{e_2} a_2 \ast ^{e_3}\ldots \ast ^{e_n} a_n\textrm{,}\\
x'&=b_0 \ast ^{e_1'}b_1 \ast ^{e_2'} b_2  \ast ^{e_3'}\ldots \ast ^{e_m'} b_m
\end{align*}
are their reduced expressions, and $a_i, b_j $ lie in $Q_1 \sqcup Q_2$. 
\vspace*{.6mm}

Case 1: $x, x' \in Q_1$ or  $x, x' \in Q_2$. Suppose that  $x, x' \in Q_1$. Since $Q_1$ is a residually finite quandle, there exist a finite quandle $F$ and a quandle homomorphism $\phi: Q_1 \rightarrow F$ such that $\phi(x)\neq \phi(x')$. Define a map $\tilde{\phi}: Q \rightarrow {F}$ by setting
$$
\tilde{\phi}(q)= \left\{
\begin{array}{ll}
\phi(q) &  \textrm{if}~q  \in Q_1,\\
a~& \textrm{if}~ q \in Q_2, ~~\textrm{where}~~a~~\textrm{is some fixed element of}~~F.
\end{array} \right.
$$
Since $\tilde{\phi}$ preserve all the relations in $Q$, it extends to a quandle homomorphism with $\tilde{\phi}(x) \neq \tilde{\phi}(x')$ in $F$.
\vspace*{.6mm}

Case 2: $ x \in Q_1$ and $x' \in Q_2$. Consider a map $\phi : Q \rightarrow FQ(X)$, where $X= \{ a, b \}$ and $FQ(X)$ is the free quandle on $X$, defined as
$$
\phi(q)= \left\{
\begin{array}{ll}
a~&  \textrm{if}~q  \in Q_1,\\
b~& \textrm{if}~ q \in Q_2.
\end{array} \right.
$$
Since $\phi$ preserve all the relations in $Q$, it extends to a quandle homomorphism with $\phi(x) \neq \phi(x')$ in $FQ(X)$.
\vspace*{.6mm}

Case 3: $ x \in Q \setminus (Q_1 \sqcup Q_2)$ and $x ' \in Q_1$. We can assume that either $a_0 \in Q_1$, $a_1 \in Q_2$ and $a_2 \ldots , a_n \in Q_1 \sqcup Q_2$ or $a_0 \in Q_2$, $a_1 \in Q_1$ and $a_2,\ldots , a_n \in Q_1 \sqcup Q_2$ i.e.,

$$
x= \left\{
\begin{array}{ll}
q_1 \ast^{e_1} q_2  \ast^{e_2} a_2 \ast ^{e_3} \ldots \ast^{e_n} a_n & ~~\textrm{where}~~q_1 \in Q_1,~~q_2 \in Q_2,\\
\textrm{or} &\\
q_2 \ast ^{e_1} q_1 \ast ^{e_2} a_2 \ast ^{e_3} \ldots \ast^{e_n} a_n & ~~\textrm{where}~~q_1 \in Q_1,~~q_2 \in Q_2.
\end{array} \right.
$$

It follows from Lemma \ref{presentation-of-associated-group-of-free-product-of-quandles}, Theorem \ref{free-product-of-residually-finite-groups} and \cite[Proposition 4.1]{Bardakov-Singh-Singh} that $\Conj\big(\As(Q)\big)$ is a residually finite quandle. Let
$$
\eta:Q  \rightarrow \Conj\big(\As(Q)\big)
$$
be the natural quandle homomorphism (see the discussion below Theorem \ref{presentation-of-associated-group}). Then, we have
$$
\eta(x_0 \ast ^{e_1}x_1  \ast ^{e_2} x_2 \ast ^{e_3}\ldots \ast ^{e_n} x_n )= (e_{x_1} ^{e_1}  e_{x_2} ^{e_2} \ldots  e_{x_n} ^{e_n})^{-1} e_{x_0} (e_{x_1} ^{e_1}  e_{x_2} ^{e_2} \ldots  e_{x_n} ^{e_n}).
$$
We claim that $\eta(x) \neq \eta(x')$. 
\vspace*{.6mm}

Subcase 3.1: If $x =  q_1 \ast^{e_1} q_2  \ast^{e_2} a_2 \ast ^{e_3} \ldots \ast^{e_n} a_n, ~~\textrm{where}~~q_1 \in Q_1,~~q_2\in Q_2,~~ \textrm{and}~~a_2, \ldots, a_{n}   \in Q_1 \sqcup Q_2$, then
\begin{align*}
\eta(x) &= (e_{q_2} ^{e_1}  e_{a_2} ^{e_2} \ldots  e_{a_n} ^{e_n})^{-1} e_{q_1} (e_{q_2} ^{e_1}  e_{a_2} ^{e_2} \ldots  e_{a_n} ^{e_n})\\
&= e_{a_n} ^{-e_n} \ldots e_{a_2} ^{-e_2} e_{q_2} ^{-e_1} e_{q_1} e_{q_2} ^{e_1}  e_{a_2} ^{e_2} \ldots  e_{a_n} ^{e_n}.
\end{align*}
Suppose that $\eta(x)=\eta(x')$. Then by the Remark \ref{image-under-eta} and the fact that elements of $\As(Q_1)$ have no relations with elements of $\As(Q_2)$ in the group $\As(Q)$, it follows that either $e_{q_2} ^{e_1}  e_{a_2} ^{e_2} \ldots  e_{a_n} ^{e_n}= 1$ in $\As(Q)$ or $e_{q_2} ^{e_1}  e_{a_2} ^{e_2} \ldots  e_{a_n} ^{e_n}= e_{q_{i_1}}^{\epsilon _1} e_{q_{i_2}}^{\epsilon _2} \ldots e_{q_{i_k}}^{\epsilon _k} $, where $q_{i_1}, q_{i_2} , \ldots, q_{i_k} $ belongs to $Q_1$ and $\epsilon_j = \pm 1$ for $1 \leq j \leq k$. Since $\As(Q)$ has a right action on the quandle $Q$, this implies that in either situation $q_1.(e_{q_2} ^{e_1}  e_{a_2} ^{e_2} \ldots  e_{a_n} ^{e_n})$ belongs to $Q_1$. Thus, $x=q_1 \ast ^{e_1} q_2 \ast ^{e_2} a_2 \ast ^{e_3}\ldots \ast ^{e_n} a_n  \in Q_1$, which is a contradiction. Hence we must have $\eta(x) \neq \eta(x')$.
\vspace*{.6mm}

Subcase 3.2: If $x = q_2 \ast ^{e_1} q_1 \ast ^{e_2} a_2 \ast ^{e_3} \ldots \ast^{e_n} a_n,  ~~\textrm{where}~~q_1 \in Q_1,~~q_2 \in Q_2~~ \textrm{and}~~a_2, \ldots , a_{n} \in Q_1 \sqcup Q_2,$ then
\begin{align*}
\eta(x) &= (e_{q_1} ^{e_1}  e_{a_2} ^{e_2} \ldots  e_{a_n} ^{e_n})^{-1} e_{q_2} (e_{q_1} ^{e_1}  e_{a_2} ^{e_2} \ldots  e_{a_n} ^{e_n})\\
&= e_{a_n} ^{-e_n} \ldots e_{a_2} ^{-e_2} e_{q_1} ^{-e_1} e_{q_2} e_{q_1} ^{e_1}  e_{a_2} ^{e_2} \ldots  e_{a_n} ^{e_n}.
\end{align*}
Clearly $\eta(x) \neq \eta(x')$ since they belong to different conjugacy classes in $\As(Q)$.
\vspace*{.6mm}

Case 4: $ x \in Q \setminus (Q_1 \sqcup Q_2)$ and $x ' \in Q_2$. This is similar to Case 3.
\vspace*{.6mm}

Case 5: $ x, x' \in Q \setminus (Q_1 \sqcup Q_2).$ This case can be reduced to one of the Cases (1--4) by repeated use of the second quandle axiom. More precisely, we can replace the element $x$ by $y$ and $x'$ by $y'$, where
\begin{align*}
y&= a_0 \ast ^{e _1} a_1 \ast ^{e_2} \cdots \ast^{e_n} a_n \ast^{-e_m'} b_m \ast^{-e' _{m-1}} b_{m-1} \ast^{-e' _{m-2}} \cdots \ast^{-e' _{1}} b_1,\\
y'&= b_0.
\end{align*}
Since finite quandles, free quandles \cite[Theorem 5.3]{Bardakov-Singh-Singh} and $\Conj\big(\As(Q)\big)$ are residually finite, we conclude that $Q= Q_1 \star Q_2$ is a residually finite quandle.
\end{proof}
%%%%%%%%%%%%%%%%%%%%%%%%%%%%%%%%%%%%%%%%%%%%%%%%%%%%%%%%%%%%%%%%%%%%%
%%%%%%%%%%%%%%%%%%%%%%%%%%%%%%%%%%%%%%%%%%%%%%%%%%%%%%%%%%%%%%%%%%%%%
%%%%%%%%%%%%%%%%%%%%%%%%%%%%%%%%%%%%%%%%%%%%%%%%%%%%%%%%%%%%%%%%%%%%%
\vspace*{.6mm}

The preceding result can be extended to arbitrary family of quandles.

\begin{theorem}\label{arbitrary-free-product-of-residually-finite-quandles}
Let $\{Q_i\}_{i \in I}$ be a family of residually finite quandles. If each $\As(Q_i)$ is a residually finite group, then the free product $\star_{i \in I} Q_i$ is a residually finite quandle.
\end{theorem}
\begin{proof}
Let $Q= \star_{i \in I} Q_i$ be the free product of residually finite quandles $Q_i$. Let $x, x' \in Q$ be two distinct elements such that
\begin{align*}
x &=a_0 \ast ^{e_1}a_1  \ast ^{e_2} a_2 \ast ^{e_3}\cdots \ast ^{e_n} a_n\textrm{,}\\
x'&=b_0 \ast ^{e_1'}b_1 \ast ^{e_2'} b_2  \ast ^{e_3'}\cdots \ast ^{e_m'} b_m.
\end{align*}
Consider the set $S = \{ a_i, b_j ~~|~~ 1 \leq i \leq n, ~~ 1 \leq j \leq m \}$. Then $S$ is a finite set contained in $ Q_{i_1} \sqcup Q_{i_2} \sqcup \cdots \sqcup Q_{i_k}$ for some $i_1, i_2, \ldots, i_k \in I$. Define a map $$\phi: Q \rightarrow Q_{i_1} \star Q_{i_2} \star \cdots \star Q_{i_k}$$ by setting

$$
\phi(q)= \left\{
\begin{array}{ll}
q &  \textrm{if}~~q  \in  Q_{i_1} \sqcup Q_{i_2} \sqcup \cdots \sqcup Q_{i_k},\\
a~& \textrm{if}~ q \in \sqcup_{i \in I} Q_i \setminus (Q_{i_1} \sqcup Q_{i_2} \sqcup \cdots \sqcup Q_{i_k}),
\end{array} \right.
$$
where $a$ is some fixed element in $\sqcup_{i \in I} Q_i \setminus ( Q_{i_1} \sqcup Q_{i_2} \sqcup \cdots \sqcup Q_{i_k})$. Since $\phi$ preserves all the relations in $Q$, it extends to a quandle homomorphism with $\phi(x) \neq \phi(x')$. Hence by Theorem \ref{finite-free-product-of-residually-finite-quandles}, $Q$ is a residually finite quandle.
\end{proof}
\vspace*{.6mm}
%%%%%%%%%%%%%%%%%%%%%%%%%%%%%%%%%%%%%%%%%%%%%%%%%%%%%%%%%%%%%%%%%%%%%
%%%%%%%%%%%%%%%%%%%%%%%%%%%%%%%%%%%%%%%%%%%%%%%%%%%%%%%%%%%%%%%%%%%%%
%%%%%%%%%%%%%%%%%%%%%%%%%%%%%%%%%%%%%%%%%%%%%%%%%%%%%%%%%%%%%%%%%%%%%
Now we present the main result of this section.

\begin{theorem}\label{split-link-res-finite}
The link quandle of any link is residually finite.
\end{theorem}

\begin{proof}
The link quandle of a split link is a free product of link quandles of its non-split components. By Theorem \ref{non-split-link-quandle-is-residually-finite}, non-split link quandles are residually finite. Now using the fact that all link groups are residually finite, the result now follows from Theorem \ref{finite-free-product-of-residually-finite-quandles}.
\end{proof}

Recall that a quandle $X$ is called Hopfian if every surjective quandle endomorphism of $X$ is injective. We conclude with the following result which is a consequence of  the preceding theorem and \cite[Theorem 5.7, Theorem 5.11, Theorem 4.3]{Bardakov-Singh-Singh}.
\vspace*{.3mm}

\begin{corollary}
The link quandle of a link is Hopfian, has solvable word problem, and has residually finite inner automorphism group.
\end{corollary}
\vspace*{.3mm}

%%%%%%%%%%%%%%%%%%%%%%%%%%%%%%%%%%%%%%%%%%%%%%%%%%%%%%%%%%%%%%%%%%%%%%%%%%%%%%%%%%%%%%%%%%%%%%%%%%%
%%%%%%%%%%%%%%%%%%%%%%%%%%%%%%%%%%%%%%%%%%%%%%%%%%%%%%%%%%%%%%%%%%%%%%%%%%%%%%%%%%%%%%%%%%%%%%%%%%%
%%%%%%%%%%%%%%%%%%%%%%%%%%%%%%%%%%%%%%%%%%%%%%%%%%%%%%%%%%%%%%%%%%%%%%%%%%%%%%%%%%%%%%%%%%%%%%%%%%%

\section{Concluding remarks}\label{concluding-remarks}
In this section, we first give an alternate proof of residual finiteness of quandles of split links whose each component is a prime knot.  We note the following result due to Ryder \cite[Corollary 3.6]{Ryder}.
\vspace*{.6mm}

\begin{theorem}\label{prime-quandles-injective-in-associated-group}
The fundamental quandle of a knot in $\mathbb{S}^3$ embeds into its associated group if and only if the knot is prime.
\end{theorem}
\vspace*{.6mm}

It is interesting to know which other quandles embeds into their associated groups. In this direction, we refer to a recent result of Bardakov and Nasybullov \cite[Lemma 7.1]{Bardakov-Nasybullov}.

\begin{lemma}\label{embedding-of-quandles-into-associated-group}
Let $Q, P$  be quandles. If the natural maps $Q \rightarrow \As(Q), P \rightarrow \As(P)$ are injective, then the natural map $Q \star P \rightarrow \As(Q) \star \As(P)$ is injective.
\end{lemma}
\vspace*{.6mm}

As a consequence of the above results, we have

\begin{theorem}
If $L$ is a link consisting of untangled components each of which is a prime knot, then $Q(L)$ is a residually finite quandle.
\end{theorem}

\begin{proof}
Observe that the link quandle $Q(L)$ of the link $L$ is a free product of knot quandles of its constituent prime knots. Further, recall that the associated group of a knot is the knot group, which is residually finite. The result now follows from Theorem \ref{prime-quandles-injective-in-associated-group}, Lemma \ref{embedding-of-quandles-into-associated-group} and Theorem \ref{free-product-of-residually-finite-groups}.
\end{proof}
\par
Finally we discuss residual finiteness of the associated groups of quandles. The following results are well-known in combinatorial group theory.
\vspace*{.6mm}

\begin{theorem}\label{finitely-presented-quotient}
If $G$ is a finitely generated group with infinitely generated center $\Z(G)$, then the quotient $G/\Z(G)$ is not finitely presented.
\end{theorem}
\vspace*{.6mm}

\begin{theorem}\label{virtually-residually-finite-groups}
Let $G$ be a group. If $N$ is a normal subgroup of finite index in $G$ and is residually finite group, then $G$ is residually finite group.
\end{theorem}

\begin{proposition}\label{ass-group-finite}
If $X$ is a finite quandle, then its associated group $\As(X)$ is a residually finite group.
\end{proposition}

\begin{proof}
Consider the natural group homomorphism  $\psi_{X}: \As(X) \rightarrow \Inn(X)$. Since $X$ is a finite quandle, the inner automorphism group $\Inn(X)$ of $X$ is finite, and hence  $\As(X)/ \Ker(\psi_{X})$ is finite. Moreover, $\Ker(\psi_{X})$ is contained in the center $\Z\big(\As(X)\big)$ of $\As(X)$, and hence $\As(X)/ \Z\big(\As(X)\big)$ is finite. By Theorem \ref{finitely-presented-quotient}, $\Z\big(\As(X)\big)$ is a finitely generated abelian group, and hence residually finite. The result now follows from Theorem \ref{virtually-residually-finite-groups}.
\end{proof}
\vspace*{.6mm}
\par
Since associated groups of finite quandles, free quandles and link quandles are residually finite, the following seems to be the case in general.
\vspace*{.6mm}

\begin{conjecture}\label{residually-finite-associated-group}
The associated group of a finitely presented residually finite quandle is a residually finite group.
\end{conjecture}
\vspace*{.6mm}

If the above conjecture is true, then by Theorem \ref{arbitrary-free-product-of-residually-finite-quandles}, a free product of finitely presented residually finite quandles is residually finite, which is an analogue of Theorem \ref{free-product-of-residually-finite-groups} for quandles.
\vspace*{.6mm}

\begin{ack}
Bardakov acknowledges support from the Russian Science Foundation project N 16-41-02006. Mahender Singh acknowledges support from INT/RUS/RSF/P-02 grant and SERB Matrics Grant MTR/2017/000018. Manpreet Singh thanks IISER Mohali for the PhD Research Fellowship.
\end{ack}
%%%%%%%%%%%%%%%%%%%%%%%%%%%%%%%%%%%%
%%%%%%%%%%%%%%%%%%%%%%%%%%%%%%%%%%%%

\end{document}